\documentclass[12pt,a4paper,reqno]{amsart}
\usepackage{amsmath,latexsym,epsfig,amsfonts,amssymb,graphicx}

\newtheorem{theorem}{Theorem}[section]
\newtheorem{lemma}[theorem]{Lemma}
\newtheorem{proposition}[theorem]{Proposition}
\newtheorem{corollary}[theorem]{Corollary}

\theoremstyle{definition}

\newtheorem{remark}[theorem]{Remark}
\newtheorem{opening}[theorem]{Question}
\theoremstyle{plain}

\newcommand{\R}{\ensuremath{\mathbb{R}}}

\newcommand{\Diffeo}{\ensuremath{\textup{Diffeo}}}

\renewcommand{\leq}{\leqslant}
\renewcommand{\geq}{\geqslant}
\renewcommand{\le}{\leqslant}

\numberwithin{equation}{section}

\title[Reversibility of diffeomorphisms]{Reversibility in the diffeomorphism group of the real line}

\author{Anthony  G. O'Farrell}
\address{Department of Mathematics, National Univeristy of Ireland Maynooth, Maynooth, County Kildare, Ireland}
\email{anthonyg.ofarrell@gmail.com}

\author{Ian Short}
\email{ims25@cam.ac.uk}
\thanks{Both authors were supported by Science Foundation Ireland grant 05/RFP/MAT0003. The first author was also supported by the ESF Network HCAA. The authors are grateful to \'Etienne Ghys for valuable
assistance.}



\date{\today}

\begin{document}

\begin{abstract}
An element of a group is said to be \emph{reversible} if it is
conjugate to its inverse. We characterise the reversible
elements in the group of diffeomorphisms of the 
real line, and in the subgroup of order preserving
diffeomorphisms. 
\end{abstract}

\maketitle

\section{Introduction}\label{S: intro}

An element of a group is \emph{reversible} if it is
conjugate to its inverse. We say that such an element is \emph{reversed} by its conjugator. A \emph{diffeomorphism} of the real line $\mathbb{R}$ is an infinitely differentiable homeomorphism of $\mathbb{R}$ whose derivative never vanishes. We consider the group $\Diffeo(\R)$ of all diffeomorphisms of  $\R$, and the subgroup $\Diffeo^+(\R)$ of order preserving  diffeomorphisms. The object of this paper is to characterise the reversible elements in each of these two groups.

An \emph{involution} in a group is an element of order two. One way to obtain a reversible element is to form the product of two involutions.  Such an element is reversed by each of the two involutions, and conversely,  an element that is reversed by an involution  can be expressed as a product of two involutions. Elements reversed by involutions are called \emph{strongly reversible}. The only involution in $\Diffeo^+(\R)$ is the identity map.  There are many non-trivial involutions in $\Diffeo(\R)$, but all are conjugate to the map $x\mapsto -x$. 

Interest in reversibility originates from the theory of time-reversible symmetry in dynamical systems, and background to the subject can be found in \cite{De1976,LaRo1998}. Finite group theorists use the terms \emph{real} and \emph{strongly real} instead of \emph{reversible} and \emph{strongly reversible}, because an element $g$ of a finite group $G$ is reversible if and only if each irreducible character of $G$ takes a real value when applied to $g$. Reversibility in the homeomorphism group of the real line has been considered before by Jarczyk \cite{Ja2002} and Young \cite{Yo1994}. See also \cite{FiSc1955,OF2004}. Reversibility in the group of invertible formal power series was considered by O'Farrell in \cite{OF2008}. Previously, in \cite{Ca1971}, Calica had studied reversibility in groups of germs of homeomorphisms and diffeomorphisms that fix $0$. Reversibility in the diffeomorphism group of the real line is of particular interest because, whilst it is difficult to fully classify conjugacy in $\Diffeo^+(\R)$ and $\Diffeo(\R)$, we are able to give a complete account of reversibility.

The main results of the paper follow.

\begin{theorem}\label{theorem-reversible-diffeo+}
An element of $\Diffeo^+(\R)$ is reversible if and only if it 
is conjugate to a
map $f$ in $\Diffeo^+(\R)$ that fixes each integer and satisfies
\begin{equation}\label{E: wave}
f(x+1) = f^{-1}(x)+1,\qquad x\in \R.
\end{equation}
\end{theorem}

Theorem~\ref{theorem-reversible-diffeo+} gives us an explicit method to generate all order preserving reversible diffeomorphisms.

\begin{theorem}\label{theorem-reversible-diffeo}
Each reversible element of $\Diffeo(\R)$ is either (i) strongly reversible,
or (ii) an element of $\Diffeo^+(\R)$ that is reversible in
$\Diffeo^+(\R)$.
\end{theorem} 

The alternatives (i) and (ii) are not exclusive. If $f$ is an order reversing reversible diffeomorphism then Theorem~\ref{theorem-reversible-diffeo} tells us that it is strongly reversible. Composing two non-trivial involutions in $\Diffeo(\R)$ gives rise to an order preserving diffeomorphism; hence $f$ must be an involution. 


We prove the following result about composites of reversible maps.

\begin{theorem}\label{T: +reversibleComposite}
Each member of $\Diffeo^+(\R)$ can be expressed as a composite of four reversible diffeomorphisms.
\end{theorem}

We do not know whether each element of $\Diffeo^+(\R)$ can be expressed as a composite of three, or even two, reversible elements. We also prove the following  result about composites of involutions.

\begin{theorem}\label{T: 4involutions}
Each member of $\Diffeo(\R)$ can be expressed as a composite of four involutions.
\end{theorem}

The number four in this theorem is sharp because each order preserving diffeomorphism that is not strongly reversible cannot be expressed as a composite of three involutions. An obvious corollary of Theorem~\ref{T: 4involutions} is that each element of $\Diffeo(\R)$ can be expressed as a composite of two (strongly) reversible diffeomorphisms.

The structure of the paper is as follows. Section \ref{S: background} contains relevant background material. Then in Section \ref{S: +} we focus on the group $\Diffeo^+(\R)$, and prove all our results related to that group, including Theorems~\ref{theorem-reversible-diffeo+} and \ref{T: +reversibleComposite}. Section \ref{S: +-} is about the group $\Diffeo(\R)$, and in that section we prove all our results about $\Diffeo(\R)$, including Theorems~\ref{theorem-reversible-diffeo} and \ref{T: 4involutions}. Finally, in Section \ref{S: open questions} we list some open problems.

\section{Background results}\label{S: background}

An element $f$ of a group $G$ is \emph{reversible} if there is another element $h$ in $G$ such that $hfh^{-1}=f^{-1}$. We say that $h$ \emph{reverses} $f$, or that $f$ is \emph{reversed} by $h$. If $g$ is an element of $G$ that is conjugate to $f$ then $g$ is also reversible. We denote the fixed point set of  a homeomorphism $g$ by $\text{fix}(g)$. Listed below are several results about diffeomorphisms of the real line. 

\begin{lemma}[Sternberg, \cite{St1957}]\label{L: fixedPointFree} 
Each fixed point free member of $\Diffeo^+(\R)$ is conjugate either to the map $x\mapsto x+1$ or the map $x\mapsto x-1$.
\end{lemma}

We remark that $x\mapsto x+1$ and $x\mapsto x-1$ are not conjugate in $\Diffeo^+(\R)$, but they are conjugate by the map $x\mapsto -x$ in $\Diffeo(\R)$.

\begin{lemma}[Kopell, Lemma 1(a), \cite{Ko1970}]\label{L: Kopell1}
Suppose that $f$ and $g$ are $C^2$ order preserving homeomorphisms of an interval $[a,b)$ such that $fg=gf$. If $g$ has no fixed points in $(a,b)$, but $f$ has a fixed point in $(a,b)$, then $f$ is the identity map.
\end{lemma}

For $f$ in $\Diffeo(\R)$, 
we use the notation $T_af$ to denote the \emph{truncated} Taylor series
\[
T_af = \sum_{n=1}^\infty \frac{f^{(n)}(a)}{n!} X^n,\
\]
regarded as a formal power series in the indeterminate $X$.

Let $\mathcal P$ denote the group of formally invertible formal
power series having real coefficients, 
under the operation of formal composition. The 
identity of $\mathcal P$ is the series
\[
X = X+0X^2+0X^3+\dotsb.
\]
The formal inverse of a power series $P$ is denoted $P^{-1}$.
\begin{lemma}[Kopell, Lemma 1(b), \cite{Ko1970}]\label{L: Kopell2}
Let $f$ and $g$ be two elements of $\Diffeo^+(\R)$ that both fix $0$ and commute. If $T_0f=X$ and $0$ is not an interior point of $\textup{fix}(f)$, then $T_0g=X$ also.
\end{lemma}

Our final lemma is about reversibility of formal power series. We are unable to find a precise reference for this lemma, so we provide a brief proof that relies on results from \cite{Lu1994,OF2008}.

\begin{lemma}\label{lemma-formal-reversible}
If a non-identity element $S=X+\dotsb$ in $\mathcal P$  is reversed by another element $T=\lambda X+\dotsb$, where $\lambda <0$, then $T$ is an involution.
\end{lemma}
\begin{proof}
Let $S=X+a_{p}X^{p}+\dotsb$, where $a_{p}\neq 0$. By \cite[Theorem 4 and Corollary 6]{OF2008}, $p$ is \emph{even}. Since $T^2$ commutes with $S$ we deduce from \cite[Proposition 1.5]{Lu1994} that the $X$ coefficient of $T^2$, namely $\lambda^2$, equals $1$. Hence $\lambda = -1$. Next, we can apply \cite[Lemma 3 (ii)]{OF2008} to deduce that either $T$ is an involution (in which case the lemma is proved), or else there is an \emph{odd} integer $q$ and  non-zero real number $b_q$ such that  $T^2=X+b_qX^q+\dotsb$. However, Lubin shows in \cite[Corollary 5.3.2(a)]{Lu1994} that two non-identity commuting power series $X+u_mX^m+\dotsb$, where $u_m\neq 0$, and $X+v_nX^n+\dotsb$, where $v_n\neq 0$, must satisfy $m=n$. We reach a contradiction because $p$ is even and $q$ is odd. Therefore $T$ is an involution.
\end{proof}

\section{Reversibility of order preserving diffeomorphisms}\label{S: +}

\subsection{Reversible maps}\label{SS: +reversible}\quad


Elementary dynamical considerations tell us that a  reversible element in the group of order preserving homeomorphisms of $\R$ must have infinitely many fixed points. In fact, a reversing conjugation  can only achieve its purpose by shunting the components of the complement of the fixed point set.  This was first pointed out by Calica \cite{Ca1971}.

We use the following lemma about homeomorphisms.

\begin{lemma}\label{L: useful1}
Suppose that $f$ and $h$ are order preserving homeomorphisms of $\mathbb{R}$ such that $hfh^{-1}=f^{-1}$. Then each fixed point of $h$ is also a fixed point of $f$.
\end{lemma}
\begin{proof}
Suppose that $h$ fixes the point $p$. We have two equivalent equations $hfh^{-1}=f^{-1}$ and $hf^{-1}h^{-1}=f$. From these equations we obtain $hf(p)=f^{-1}(p)$ and $hf^{-1}(p)=f(p)$. Order preserving homeomorphisms such as $h$ have no periodic points other than fixed points; thus $f(p)=f^{-1}(p)$.  Therefore $f^2(p)=p$, which means that $f(p)=p$.
\end{proof}

The next lemma works for diffeomorphisms, but not for homeomorphisms. 

\begin{lemma}\label{L: useful2}
Suppose that $f$ and $h$ are order preserving diffeomorphisms of $\mathbb{R}$ such that $hfh^{-1}=f^{-1}$. If $h$ has a fixed point then $f$ is the identity map.
\end{lemma}
\begin{proof} 
If $h$ is the identity map then $f$ is an order preserving involution, and therefore $f$  is also the identity map. Suppose then that $h$ is not the identity map, but that it nevertheless has a fixed point. Choose a component $(a,b)$ in the complement of $\text{fix}(h)$. One of $a$ or $b$ is a real number (that is, we cannot have both $a=-\infty$ and $b=+\infty$). Let us assume that $a$ is a real number; the other case can be dealt with similarly.  By Lemma~\ref{L: useful1}, $f$ fixes $(a,b)$ as a set. The map $f$ cannot be free of fixed points on $(a,b)$. To see this, suppose, by switching $f$ and $f^{-1}$ if necessary, that $f(x)>x$ for each real number $x$ in $(a,b)$. Then
\[
x>f^{-1}(x)=hf(h^{-1}(x))>hh^{-1}(x)=x,
\] 
which is a contradiction. Since $f$ has a fixed point in $(a,b)$,
Lemma~\ref{L: Kopell1} applied to the maps $f$ and $h^2$ shows that $f$ coincides with the identity map on $(a,b)$. We already know that $f$ coincides with the identity map on $\text{fix}(h)$; thus $f$ is the identity map.
\end{proof}

We are now in a position to prove the first main result.

\begin{proof}[Proof of Theorem~\ref{theorem-reversible-diffeo+}]
Let $t$ be the map given by $t(x)=x+1$ for each $x$. Then \eqref{E: wave} states that $tft^{-1}=f^{-1}$. Thus a diffeomorphism $f$ that satisfies \eqref{E: wave} is reversible, and likewise all conjugates of $f$ are reversible.

Conversely, suppose that $g$ and $h$ are elements of $\Diffeo^+(\R)$ such that $hgh^{-1}=g^{-1}$. The theorem holds when $g$ is the identity, so let us suppose that $g$ is not the identity. This means that we can assume, by Lemma~\ref{L: useful2}, that $h$ is free of fixed points. Observe that  $h^{-1}gh=g^{-1}$; hence by replacing $h$ with $h^{-1}$ if necessary we can assume that $h(x)>x$ for each $x$. 

By Lemma~\ref{L: fixedPointFree} we see that there is an element $k$ in $\Diffeo^+(\R)$ such that $khk^{-1}=t$. Define $f=kgk^{-1}$. Then $tft^{-1}=f^{-1}$; that is, \eqref{E: wave} holds.  Now $f$, like $g$, must have a fixed point, and by conjugating $f$ by a translation we may assume that this fixed point is $0$. Since translations commute, this conjugation does not affect \eqref{E: wave}. Finally, from the equation $tf^nt^{-1}=f^{-n}$ we deduce that $f$ fixes each integer.
\end{proof}

We can construct a diffeomorphism $f$ that satisfies \eqref{E: wave} explicitly by defining $f$ on $[0,1]$ to be an arbitrary order preserving diffeomorphism of $[0,1]$ such that $T_0f=(T_1f)^{-1}$, and then extending the domain of $f$ to $\mathbb{R}$ using \eqref{E: wave}. More precisely, we have the following corollary of Theorem~\ref{theorem-reversible-diffeo+}.

\begin{corollary} Let $g\in\Diffeo^+(\R)$. Then the following two 
conditions are equivalent:
\begin{enumerate}
\item The map $g$ is reversible in $\Diffeo^+(\R)$.
\item There exist
\begin{enumerate}
\item a formally invertible power series $P$;
\item an order preserving  diffeomorphism
$\phi$ of $[0,1]$, with 
\[
 T_0\phi=P,\qquad T_1\phi=P^{-1} ;
\]
and
\item $h\in\Diffeo^+(\R)$,
such that $g= hfh^{-1}$, where, for each $n$ in $\mathbb{Z}$,
\[
f(x) =
\begin{cases}
\phi(x-2n)+2n& \text{if   \,$2n\le x<2n+1$,}\\
\phi^{-1}(x-2n-1)+2n+1& \text{if   \,$2n+1\le x<2n+2$.}
\end{cases}
\]
\end{enumerate}
\end{enumerate}
\end{corollary}

\begin{remark}
Each map $f$ of part (ii) commutes with $x\mapsto x+2$. Hence $f$ is the lift under the covering map $x\mapsto \exp(\pi i x)$ of the order preserving diffeomorphism  of the unit circle $\bar{f}$ given by $\bar{f}(e^{i\pi\theta})=e^{i\pi f(\theta)}$. Moreover $f$ is reversed by $x\mapsto x+1$; thus $\bar{f}$ is reversed by rotation by $\pi$.
\end{remark}

\subsection{Composites of reversible maps}\label{SS: +compositesReversible}\quad

\begin{lemma}\label{L: +reversibleComposite}
Each fixed point free element of $\Diffeo^+(\R)$ can be expressed as a composite of two reversible elements of $\Diffeo^+(\R)$.
\end{lemma}
\begin{proof}
By Lemma~\ref{L: fixedPointFree} it suffices to find a single fixed point free map that can be expressed as a composite of two reversible diffeomorphisms. Let $f$ be a reversible  order preserving diffeomorphism such that $f(x+1)=f^{-1}(x)+1$ for each real number $x$, and $f(y)>y$ for each element $y$ of $(0,1)$. The graph of such a map $f$ is shown in Figure~\ref{F: reversibles}.

\begin{figure}[ht]
\centering
\includegraphics[scale=1.0]{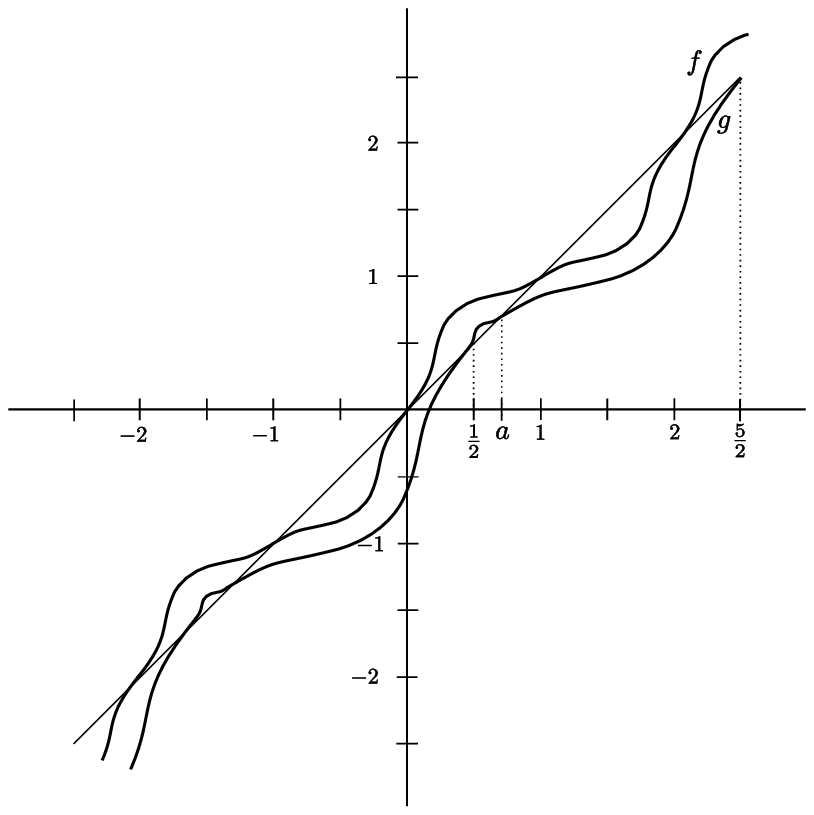}
\caption{}
\label{F: reversibles}
\end{figure}

Let $a$ be an element from the interval $\left(\frac12,f\left(\frac12\right)\right)$. Notice that every order preserving diffeomorphism $h$ of $\left[\frac12,a\right]$ satisfies $h(x)<f(x)$ for $x\in\left[\frac12,a\right]$. Choose an order preserving diffeomorphism $g$ of $\left[a,\frac52\right]$ such that $T_a g = T_{\frac52}g = X$, and such that $g(x)<f(x)$ for each $x\in \left[a,\frac52\right]$.  (This construction is possible by a classic result of Borel, which says that to each formal power series $P$ there corresponds a smooth function $f$ defined in a neighbourhood of $0$ such that $T_0f=P$.) Next, choose an order preserving diffeomorphism $k$ from $\left[\frac12,a\right]$ to $\left[a,\frac52\right]$ such that $T_{\frac12}k=T_{a}k= X$. 

We extend the definition of $g$ to $\mathbb{R}$ by defining $g(x)=k^{-1}g^{-1}k(x)$ for $x\in\left[\frac12,a\right]$, and $g(x+2)=g(x)+2$ for all $x\in\mathbb{R}$. We  extend the definition of $k$ by defining $k(x)=k^{-1}(x)+2$ for $x\in\left[a,\frac52\right]$ and $k(x+2)=k(x)+2$ for all $x\in\R$. The resulting maps $g$ and $k$ are both order preserving diffeomorphisms. Moreover, one can check that the equation $g(x)=k^{-1}g^{-1}k(x)$ is satisfied for points $x$ in $\left[\frac12,\frac52\right]$. Since both maps commute with $x\mapsto x+2$, this equation is satisfied throughout $\mathbb{R}$. Finally, we have defined $g$ such that $f(x)>g(x)$ for elements $x$ of  $\left[\frac12,\frac52\right]$, and in fact $f(x)>g(x)$ everywhere, again, because both maps commute with $x\mapsto x+2$. Therefore $g^{-1}f$ is a fixed point free diffeomorphism expressed as a composite of two reversible maps.
\end{proof}

\begin{proof}[Proof of Theorem~\ref{T: +reversibleComposite}]
Choose $f$ in $\Diffeo^+(\R)$. Choose a fixed point free diffeomorphism $g$ such that $g(x)<f(x)$ for each $x\in\mathbb{R}$. Then $g^{-1}f(x)>x$ for each $x$ in $\mathbb{R}$, so the map $h=g^{-1}f$ is also free of fixed points. Since $f=gh$, the result follows from Lemma~\ref{L: +reversibleComposite}.
\end{proof}

We do not know whether  each element of $\Diffeo^+(\R)$ is the composite 
of three reversible elements.

\section{Reversibility in the full diffeomorphisms group}\label{S: +-}

\subsection{Order reversing reversible maps }\label{SS: +-reversible}\quad

We denote the set of order reversing diffeomorphisms
of $\R$ by $\Diffeo^-(\R)$. The next proposition fails for homeomorphisms.

\begin{proposition}\label{P: -reversible}
An order reversing member of $\Diffeo(\R)$ is reversible in $\Diffeo(\R)$ if and only if it is an involution.
\end{proposition}
\begin{proof}
Involutions are all reversible by the identity map. Conversely, suppose that $f\in\Diffeo^-(\R)$, $h\in\Diffeo(\R)$, and $hfh^{-1}=f^{-1}$. By replacing $h$ with $hf$ if necessary, we may assume that $h$ preserves order. From the equation $hf=f^{-1}h$ we deduce that $h$ fixes the unique fixed point of $f$. Now, $hf^2h^{-1}=f^{-2}$, and $f^2$ preserves order; therefore Lemma~\ref{L: useful2} applies to show that $f^2$ is the identity map, as required.
\end{proof}

Proposition~\ref{P: -reversible} accounts for all 
order reversing reversible diffeomorphisms. 
In Theorem~\ref{theorem-reversible-diffeo+} 
we described all order preserving diffeomorphisms that are 
reversed by order preserving maps. That leaves only 
order preserving diffeomorphisms that are reversed by 
order reversing maps. These are examined 
next. 

\subsection{Strongly reversible maps}\label{SS: +-strongReverse}\quad

\begin{lemma}\label{L: fixedFreeReverse}
Fixed point free diffeomorphisms are strongly reversible.
\end{lemma}
\begin{proof}
A fixed point free diffeomorphism is, by Lemma~\ref{L: fixedPointFree}, conjugate in the group $\Diffeo(\R)$ to $x\mapsto x+1$, and this map is reversed by the involution $x\mapsto -x$.
\end{proof}

\begin{proposition}\label{P: +-strongReverse}
If a member of $\Diffeo^+(\R)$ is reversed by a member of $\Diffeo^-(\R)$ then it is strongly reversible.
\end{proposition}
\begin{proof}
Let $f\in\Diffeo^+(\R)$, $h\in\Diffeo^-(\R)$, and $hfh^{-1}=f^{-1}$. We wish to show that
$f$ is strongly reversible. Given Lemma~\ref{L: fixedFreeReverse}, we may assume that $f$ has a fixed point. By conjugation, we can assume that the fixed point of $h$ is $0$. Notice that $h$ permutes the fixed points of $f$. 
We define an involutive homeomorphism $k$ by
\[
k(x) = 
\begin{cases}
h(x) & \text{if $x\geq 0$}\,,\\
h^{-1}(x) & \text{ if $x<0$}\,. 
\end{cases}
\]

If $f(0)=0$ then clearly, $kfk=f^{-1}$; however, $k$ may not be a diffeomorphism.  Note that $f^\prime(0)=1$, because $f^\prime(0)=(f^{-1})^\prime(0)$.  We consider three cases.

First, suppose that $f$ coincides with the identity 
on a neighbourhood of $0$. In this case we have freedom to adjust the definition of $k$  near $0$ so that it is an involutive 
diffeomorphism, without disturbing the validity of the equation 
$kfk=f^{-1}$.

Second, suppose that $0$ is not an interior fixed point of 
$f$, but that $T_0f=X$. Since $h^2$ commutes with $f$, 
it follows from Lemma~\ref{L: Kopell2} that  $T_0h$ is an involution,
so that $k$ is already a diffeomorphism.

Third, suppose that $0$ is a fixed point of $f$
and $T_0f\not=X$.  By Lemma~\ref{lemma-formal-reversible}, $T_0h$ is an involution, and again, $k$ is a diffeomorphism.

Now suppose that $0$ lies inside a component $(a,b)$ of $\mathbb{R}\setminus\text{fix}(f)$. Since $f$ has a fixed point, we know that $(a,b)\neq\mathbb{R}$. Moreover, because the order reversing map $h$ fixes $(a,b)$, both end points $a$ and $b$ are finite, and $h(a)=b$ and $h(b)=a$.  Therefore $h^2$ fixes $a$, $b$, and $0$, and commutes with $f$. By Lemma~\ref{L: Kopell1}, 
$h^2(x)=x$ for each $x\in [a,b]$. This means that 
$h$ and $h^{-1}$ coincide inside $(a,b)$, so that $k$ is a 
diffeomorphism, and $kfk^{-1}=f^{-1}$.
\end{proof}

\begin{proof}[Proof of Theorem \ref{theorem-reversible-diffeo}]
Combine Propositions \ref{P: -reversible}
and \ref{P: +-strongReverse}.
\end{proof}

Since all non-trivial involutions in  $\Diffeo(\R)$ are conjugate to $x\mapsto -x$ we have the following explicit method to construct all strongly reversible elements of $\Diffeo(\R)$:

\begin{corollary} 
Let $g\in\Diffeo^+(\R)$. Then the following two 
conditions are equivalent:
\begin{enumerate}
\item The map $g$ is reversible in $\Diffeo(\R)$ by an order reversing diffeomorphism.
\item There exist
\begin{enumerate}
\item a formally invertible power series $P$ that is strongly reversed by the power series $-X$;
\item a point $p$ and an order preserving diffeomorphism $\phi:[p,\infty)\rightarrow [-p,\infty)$ such that $T_p\phi=P$; 
\item $h\in\Diffeo(\R)$,
such that $g= hfh^{-1}$, where
\[
f(x) =
\begin{cases}
\phi(x)  & \text{if   \,$x\geq p$,}\\
-\phi^{-1}(-x)& \text{if   \,$x\leq p$.}
\end{cases}
\]
\end{enumerate}
\end{enumerate}
\end{corollary}

Note that the graph of a map reversed by $x\mapsto -x$ is symmetric in the line $y=-x$. Refer to \cite{Ja2002,OF2004,Yo1994} for more information on strong reversibility of homeomorphisms.


Proposition~\ref{P: +-strongReverse} shows that elements of $\Diffeo^+(\R)$ that are reversed by order \emph{reversing} elements of $\Diffeo(\R)$ are strongly reversible in $\Diffeo(\R)$. There are, however,  elements of $\Diffeo^+(\R)$ that are reversed by order \emph{preserving} elements of $\Diffeo(\R)$ that are not strongly reversible in $\Diffeo(\R)$.  In fact, for order preserving diffeomorphisms, the properties of being reversible in $\Diffeo^+(\R)$ and strongly reversible in  $\Diffeo(\R)$ are logically independent. To demonstrate this, we must, in turn, find an example of an order preserving diffeomorphism that is  
\begin{enumerate}
\item neither reversible in $\Diffeo^+(\R)$ nor strongly reversible in $\Diffeo(\R)$;
\item not reversible in $\Diffeo^+(\R)$, but strongly reversible in $\Diffeo(\R)$;
\item  reversible in $\Diffeo^+(\R)$, but not strongly reversible in $\Diffeo(\R)$;
\item reversible in $\Diffeo^+(\R)$ and strongly reversible in $\Diffeo(\R)$.
\end{enumerate}
Examples of (i) and (ii) are readily constructed. For (ii), any non-trivial strongly reversible diffeomorphism which coincides with the identity map outside a compact set will suffice, because Theorem~\ref{theorem-reversible-diffeo+} tells us that such a map cannot be reversible in $\Diffeo^+(\R)$. We now give an example of (iii), and then a non-identity example of (iv).

\subsection{Example (iii)}\quad

 We shall describe an order preserving diffeomorphism $f$ that is 
reversible by order preserving diffeomorphisms, 
but not by order reversing involutions. 
The map described is not even strongly reversible as a homeomorphism. We assume some common knowledge of conjugacy in the homeomorphism group of the real line, which can be found, for example, in \cite{FiSc1955}.

We shall define $f$ to be an element of 
$\Diffeo^+(\R)$ such that $\text{fix}(f)=\mathbb{Z}$. To 
specify $f$ up to topological conjugacy, it remains 
only to describe the signature on $\mathbb{R}\setminus\mathbb{Z}$, 
which we represent by an infinite sequence of $+$ and $-$ symbols. A $+$ symbol corresponds to an interval $(n,n+1)$ for which $f(x)>x$ for each $x\in (n,n+1)$, and a $-$ symbol corresponds to an interval $(n,n+1)$ for which $f(x)<x$ for each $x\in (n,n+1)$.  The signature of a homeomorphism of $\mathbb{R}$ is discussed in more detail in \cite{FiSc1955}.
Suppose the signature of $f$ consists of  the $12$ symbol sequence
\begin{equation}\label{E: sequence}
+,+,+,-,-,+,-,-,-,+,+,-,
\end{equation}
repeated indefinitely, in both directions. The map $f$ can be chosen to be a diffeomorphism. A portion of a graph of such 
a function is shown in Figure~\ref{F: reverse}. 
It satisfies $hfh^{-1}=f^{-1}$, where $h$ is given by the 
equation $h(x)=x+6$. On the other hand, it is straightforward to see (or refer to \cite{FiSc1955}) that $f$ is not reversible by a non-trivial 
involution, as the doubly infinite sequence generated by \eqref{E: sequence}  read forwards is different from the same sequence read backwards.

\begin{figure}[ht]
\centering
\includegraphics[scale=1.0]{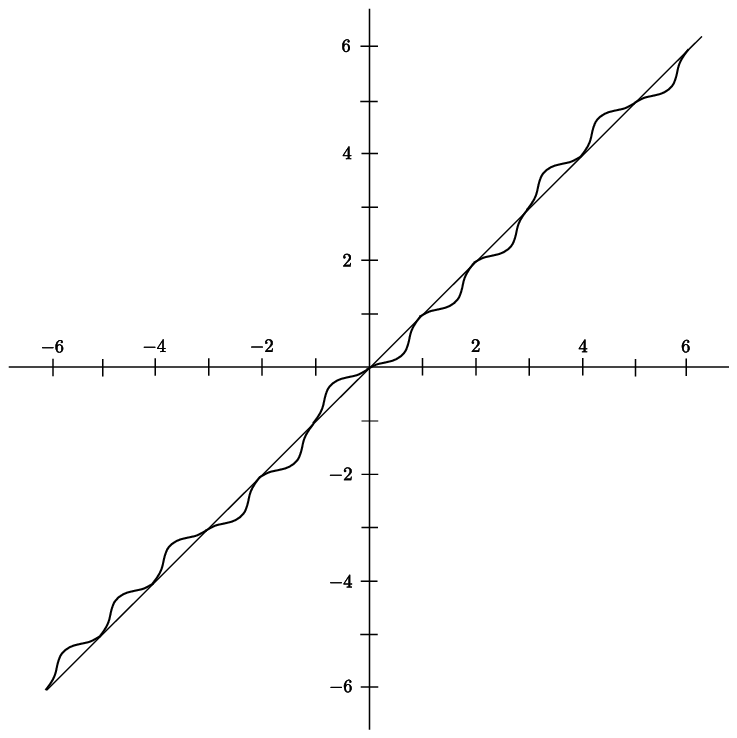}
\caption{}
\label{F: reverse}
\end{figure}

\subsection{Example (iv)}\quad

Now we give an example of a non-trivial order preserving diffeomorphism $f$ that is reversible 
in $\Diffeo^+(\R)$ \emph{and} strongly reversible in $\Diffeo(\R)$.

Choose any non-identity function $\phi$ in $\Diffeo^+([0,1])$, with
$T_0\phi=T_1\phi=X$, and define in turn
\[
\begin{array}{rclr}
\tau(x) &=& -\phi(x),&\ 0\le x\le1,\\
\tau(x) &=& \tau^{-1}(x),&\ -1\le x<0,\\
\tau(x+2) &=& -\tau(-x)-2,&\ -1<x\le1,
\end{array}
\]
and extend $\tau$ to $\R$ by requiring that
\[
\tau(x+4) =\tau(x)-4,\qquad x\in\R.
\]
Then $\tau$ is an involutive element of $\Diffeo^-(\R)$, so that $f=-\tau$ is an element of $\Diffeo^+(\R)$ that is strongly reversible in $\Diffeo(\R)$.

\begin{figure}[ht]
\centering
\includegraphics[scale=1.0]{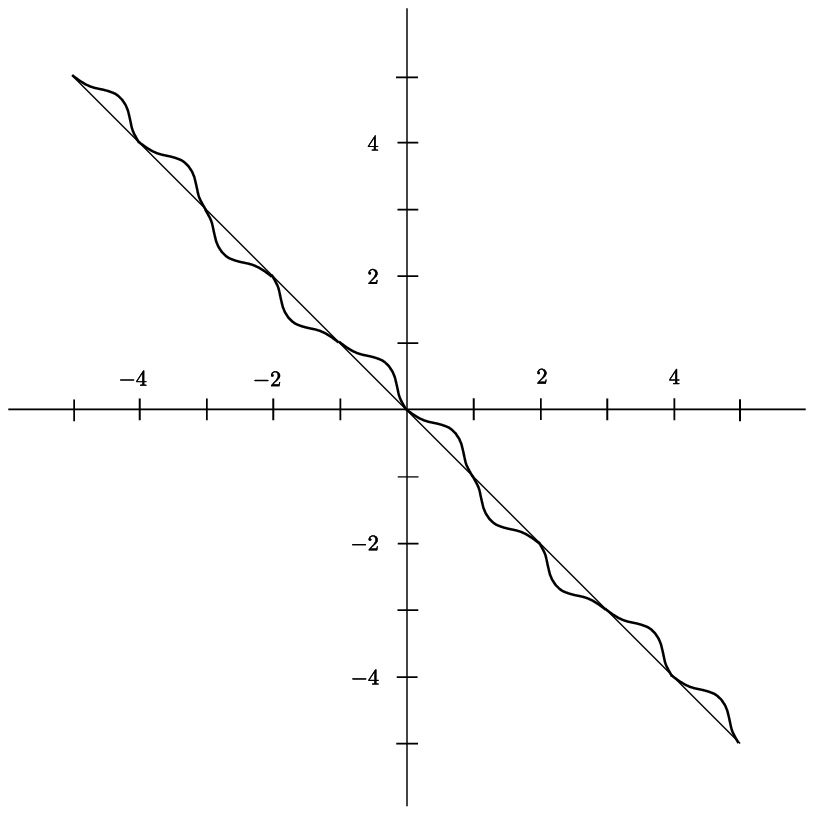}
\caption{}
\label{F: two-way-reverse}
\end{figure}

On the other hand, $\tau(x+2) = -\tau(-x)-2$ for all $x\in\R$, so
\[
f(x+2) = -\tau(x+2) = \tau(-x)+2 = f^{-1}(x)+2.
\]
Hence $f$ is also reversible in $\Diffeo^+(\R)$.

\subsection{Composites of involutions}\label{SS: +-compositesInvolutions}\quad

\begin{proposition}\label{P: 3involutions}
Each member of $\Diffeo^-(\R)$ can be expressed as a composite of three involutions.
\end{proposition}
\begin{proof}
Given an element  $f$ of $\Diffeo^-(\R)$, choose an involution $\tau$ such that $\tau(x)>f(x)$ for each real number $x$. Then $\tau f(x)>x$. The map $\tau f$ is strongly reversible, by Lemma~\ref{L: fixedFreeReverse}. Therefore $f$ can be expressed as a composite of three involutions. 
\end{proof}

\begin{proof}[Proof of Theorem~\ref{T: 4involutions}]
Given an element $f$ in $\Diffeo^+(\R)$, the map $-f$, a member of $\Diffeo^-(\R)$, can be expressed as a composite of three involutions. The result follows immediately from Propositioni~\ref{P: 3involutions}.
\end{proof}

\subsection{Composites of reversible maps}\label{SS: +-compositesReversible}\quad

\begin{proposition}
Each member of $\Diffeo(\R)$ can be expressed as a composite of two strongly reversible maps.
\end{proposition}
\begin{proof}
Immediate from Proposition~\ref{P: 3involutions} and Theorem~\ref{T: 4involutions}.
\end{proof}

The corresponding result for homeomorphisms is due to
Fine and Schweigert \cite{FiSc1955}.


\section{Open questions}\label{S: open questions}

We list two open problems which have emerged from our study.

\begin{opening}
What is the smallest positive integer $m$ such that each member of $\Diffeo^+(\R)$ can be expressed as a composite of $m$ reversible maps?
\end{opening}

Certainly $m>1$, and we have proven that $m\leq 4$. 

\begin{opening}
What can be said about reversibility in the group of homeomorphisms of the real line that are $n$ times continuously differentiable in both directions, or in the group of homeomorphisms of the real line that are real analytic in both directions?
\end{opening}

\section{Acknowledgements}

The authors thank the referees for useful comments.

\end{document}